\documentclass[10pt]{amsart}
\usepackage{amsmath,amscd,latexsym,verbatim,amssymb}
\usepackage{times}  
\usepackage[all]{xy}

 \newcommand{\resp}{{\it resp.} }
\newcommand{\cf}{{\it cf.} }
\newcommand{\ie}{{\it i.e.} }
\newcommand{\eg}{{\it e.g.} }

\newcommand{\N}{\mathbb{N}}

  \newcommand{\Z}{\mathbb{Z}}

  \newcommand{\sA}{{\mathcal{A}}}
\newcommand{\sB}{{\mathcal{B}}}

\newcommand{\sK}{{\mathcal{K}}}

\newcommand{\inj}{\hookrightarrow}

\newcommand{\surj}{\rightarrow\!\!\!\!\!\rightarrow}

\font\sm=cmr10 at 10pt

 \newcounter{spec}

\swapnumbers

\newtheorem{thm}{Theorem}[subsection]

\newtheorem{prop}[thm]{Proposition}

\numberwithin{equation}{section}

  at12pt

\renewcommand{\qed}{\hfill $\square$\medskip}
  
\begin{document}

 \title[Weak functoriality of Cohen-Macaulay algebras]{Weak functoriality of Cohen-Macaulay algebras}

\author{Yves Andr\'e}

\address{Institut de Math\'ematiques de Jussieu\\  4 place Jussieu, 75005
Paris\\France.}
\email{yves.andre@imj-prg.fr}
 %  \date{June 21, 2018}

  \begin{abstract} We prove the weak functoriality of (big) Cohen-Macaulay algebras, which controls the whole skein of ``homological conjectures" in commutative algebra \cite{H1}\cite{HH2}. 
    Namely, for any local homomorphism $ R\to R'$ of complete local domains, there exists a compatible homomorphism between some Cohen-Macaulay $R$-algebra and some Cohen-Macaulay $R'$-algebra. 
  
  When $R$ contains a field, this is already known \cite[3.9]{HH2}. When $R$ is of mixed characteristic, our strategy of proof is reminiscent of G. Dietz's refined treatment \cite{D} of weak functoriality of Cohen-Macaulay algebras in characteristic $p$; in fact, developing a ``tilting argument" due to K. Shimomoto, we combine the perfectoid techniques of \cite{A1}\cite{A2} with Dietz's result.    
   \end{abstract}

   \maketitle
  \let\languagename\relax

     \begin{sloppypar}

      \section{Introduction.} 
    
    \subsection{}  If non-noetherian rings have found their place in standard commutative algebra, it is not so much as pathological examples or products of an inexorable generalization process, but rather as invaluable auxiliaries when the ressources of classical Cohen-Macaulay theory fail. Indeed, in such a case, ``unwanted relations" between parameters of a noetherian local ring $R$ which do not form a regular sequence may still be trivialized in some big (\ie possibly non-noetherian) $R$-algebra. This accounts for the importance of the existence of (big) Cohen-Macaulay $R$-algebras in commutative algebra (\cf \cite{H1}\cite{HH2}\cite{Hu}\cite{A3} for some overviews; in the sequel, we drop the epithet big).  
     Their existence has been known and used for a long time when $R$ contains a field, but it was proven only recently in mixed characteristic using methods from $p$-adic Hodge theory (perfectoid spaces) \cite{A2}. 
  
 For more sophisticated applications to homological conjectures in commutative algebra, more is required: a weakly functorial behaviour of these Cohen-Macaulay algebras; namely, it is expected that for any local homomorphism $ R\stackrel{f}{\to} R'$ of complete local domains, there exists a compatible homomorphism between some Cohen-Macaulay $R$-algebra and some Cohen-Macaulay $R'$-algebra. 
   Again, this has been known for a long time when $R$ contains a field \cite{HH2}, and in characteristic $p$, it actually suffices to consider a lifting $ R^+ \stackrel{f^+}{\to}  R'^+$ of $f$ to absolute integral closures \cite{HH1}.  
   
      In this paper, we establish this weak functoriality in general:
      
      \begin{thm}\label{T1} Any local homomorphism $  R  {\to} R'$ of {complete} noetherian local domains fits into a commutative square  
     \begin{equation}\label{CC1}   \xymatrix @-1pc {    R  \ar[d]   \,  \ar@{->}[r]     & R'  \ar[d]    \\ C  \,  \ar@{->}[r]   & C'       }  \end{equation} 
  where $C$ and $C'$ are 
Cohen-Macaulay algebras for $R$ and $R'$ respectively. 
\end{thm}

   \subsection{}  In the remaining case to be treated, \ie when $R$  is of mixed characteristic, we prove the following more precise result:
           \begin{thm}\label{T2}  {Any local homomorphism ${f}: R  {\to} R'\,$ of {complete} noetherian local domains, with $R$ of mixed characteristic $(0,p)\,$, fits into a commutative square} 
     \begin{equation}\label{CC2}   \xymatrix @-1pc { \\    R  \ar[d]   \,  \ar@{->}[r]^{{f}}    & R'  \ar[d]      
  \\  \;\, R^+  \ar[d]   \,  \ar@{->}[r]^{{f}^+}    & \;\, R'^+  \ar[d]   
      \\ C  \,  \ar@{->}[r]    & C'    }  \end{equation} 
 {where $R^+$ and $R'^+$ are the absolute integral closures of $R$ and $R'$ respectively,  $C$ and $C'$ are 
 Cohen-Macaulay algebras for $R$ and $R'$ respectively, $C$ is $\hat\sK^{\sm o}$-perfectoid, and so is $C'$ if $R'$ is of mixed characteristic $(0,p)\,$ (\resp $C'$ is perfect if $R'$ is of characteristic $p$). }

 Moreover,  ${f}^+$ can be given in advance.
  \end{thm}
  Here $\hat\sK^{\sm o}$ is the perfectoid ring obtained from the Witt ring of the algebraic closure of the residue field of $R'$ by adjoining $p^{th}$-power roots of $p$ and completing; ``${f}^+$ can be given in advance" means that for any choice of the upper commutative square, one can form a commutative lower square as indicated.  
 
     \subsection{} Our strategy of proof is inspired by Dietz's refined weak functoriality of Cohen-Macaulay algebras in characteristic $p$. Using Cohen factorizations, one reduces to the case of a surjection, in which case one proves a stronger version of weak functoriality where $C$ is given in advance (and can be constructed for instance as in \cite{A2}): this allows to compose or decompose the morphism $f$, hence to reduce to the case when $f$ is a quotient map by a prime ideal of height one (and $R$ is normal)\footnote{in this special situation, weak functoriality has already been established in \cite{HM} in mixed characteristic, using perfectoid methods after \cite{A2} and getting important corollaries such as the vanishing conjecture on Tor's. The main progress in the present paper in this special situation lies, as said above, in the fact that $C$ may be given in advanced.}.
     
     To deal with this special case in the char. $p$ situation, Dietz uses a subtle Frobenius argument. In order to transpose it somehow to mixed characteristic, we take advantage of a remarkable insight of K. Shimomoto \cite{Sh2} who combined perfectoid techniques with Dietz's results through tilting, to show the existence of a {\it perfectoid} Cohen-Macaulay $R$-algebra $C$ (in a weak sense), so that perfectoid techniques can be applied to $C$ itself.  
     
 \bigskip 
      
 {  \small {\it Acknowledgment}. I am very grateful to K. Shimomoto for sending me his paper \cite{Sh2}, which not only supplied a useful tool but also directed my attention towards Dietz's work. I also thank him, together with L. Ma, for comments on the first draft of this paper.
 
 I am very grateful to O. Gabber for pointing out an unclear non-(almost)-triviality issue in a previous version, suggesting the use of solid closure theory to settle this issue, \cf \ref{P0} (2) (he also proposed a different solution to this technical problem), and sending several further remarks.}

  \bigskip \section{Review of absolute integral closure, (almost) perfectoid algebras and (almost) Cohen-Macaulay algebras.}
    
        \subsection{Weak functoriality of absolute integral closure} (See also \cite[3]{Hu}). 
         Let $R$ be a domain with fraction field $K$,  $ K^+$ an algebraic closure of $K$, and $R^+$ the integral closure of $R$ in $K^+$ (so that $R^+ \cap K $ is the normalization of $R$).  The absolute automorphism group  $G_K = {\rm{Aut}}_K K^+ $ may be identified with the group of $R$-automorphisms of  $R^+$.
            
       Let $R'$ be another domain with fraction field $K'$ of characteristic $0$,  $ K'^{+}$ an algebraic closure of $K$, and $R'^{+}$ the integral closure of $R'$ in $K'^{+}$. 
            
         \begin{prop}\label{P1} \begin{enumerate} \item Any homomorphism $R  \stackrel{{f}}{\to} R'$ lifts to a homomorphism $R^+  \stackrel{{f}^+}{\to} R'^+$.   If ${f}$ is injective (\resp surjective), so is ${f}^+$. 
 
 \item Assume that ${f}$ is injective or $R$ is normal. Then ${f}^+$ is unique up to precomposition by an element of $G_K$.
  
\item Assume that $R$ is normal. Any factorization $R\stackrel{{f}_1}{\to} R_1\to \cdots \to R_{n-1} \stackrel{{f}_n}{\to} R' $ of ${f}$ lifts to a factorization $R^+ \stackrel{{f}^+_1}{\to} R_1^+\to \cdots \to R_{n-1}^+ \stackrel{{f}_n^+}{\to} R'^+ $ of any given ${f}^+$. \end{enumerate}
      \end{prop}
 
 \begin{proof} $(1)$ Let us first assume that ${f}$ is {\it injective}. We note that any extension ${f}^+$ is injective \cite[Cor. 2 to Prop. 1,  \S2 n.1]{B}.  The embedding $K\inj K'$ extends to an embedding $K^+ \inj K'^+$ (unique up to precomposition by an element of $G_K$); therefore ${f}$ extends to an embedding $R^+ \inj R'^+$ (unique up to precomposition by an element of $G_K$). 
 
 Let us next assume that ${f}$ is {\it surjective}, with kernel $\frak p$. Let $\frak p^+$ be a prime ideal above $\frak p$ in the integral extension $R^+$. Every monic polynomial $P$ with coefficients in $R^+/\frak p^+$, lifts to a monic polynomial with coefficients in $R^+$, which is a product of linear factors in the absolutely integrally closed domain $R^+$, hence $P$ is a product of linear factors in $R^+/\frak p^+$ 	as well. We conclude  that the fraction field of $R^+/\frak p^+$ is algebraically closed. Since $R^+/\frak p^+$ is integral over $R'= R/\frak p$, it is isomorphic to $R'^+$, hence ${f}$ extends to a surjection $R^+ \inj R'^+$. 
 
The general case is obtained by factoring ${f}$ into a surjection followed by an embedding.

 \smallskip   \noindent $(2)$ The case when ${f}$ is injective has already been treated, and the case when $R$ is normal (\ie $R= R^+\cap K$) follows from \cite[Cor. 1 to Prop. 6,  \S2 n.1]{B}. 
 
 \smallskip   \noindent   $(3)$ By $(2)$, for any choice of lifts ${f}_i^+$, there exists $\sigma\in G_K$ such that ${f}_n^+ \circ \cdots {f}_1^+\circ \sigma = {f}^+$. It then suffices to replace ${f}_1^+$ by ${f}_1^+\circ \sigma$. 
   \end{proof}
 
 Using \cite[V, \S 2, ex. 13]{B}, one sees that $(2)$ and $(3)$ may fail if the assumption is removed (this is related to the failure of going-down for non-normal rings).
   
      \subsection{Perfectoid algebras} (\cf \cite{Sc}, and further \cite{A1}; we use a specific perfectoid valuation ring in this paper but the general theory would work over any perfectoid valuation ring). Unless otherwise specified, we denote by $\,\widehat{\,\,\,}\,$  the $p$-adic (separated) completion of any ring. We denote by $F$ the Frobenius endomorphism of any ring of characteristic $p$.

  \subsubsection{}   Let $k$ be a perfect field of characteristic $p$, $W:= W(k)$ its Witt ring, $\sK^{\sm o} := W[p^{\frac{1}{p^\infty}}]$, and $\hat\sK^{\sm o}$ its completion, which is a {\it perfectoid valuation ring} ($F$ induces an isomorphism $\hat\sK^{\sm o}/p^{\frac{1}{p}}\stackrel{\sim}{\to} \hat\sK^{\sm o}/p$) with residue field $k$; this is the valuation ring of the perfectoid field $\hat\sK := \hat\sK^{\sm o}[\frac{1}{p}]$.

 The {\it tilt} of $ \hat\sK^{\sm o}$ is defined to be $\hat\sK^{\flat\sm o}:= \lim_{F}\,\hat\sK^{\sm o}/p$. This is a perfect complete valuation ring of characteristic $p$. In fact, denoting by $p^\flat$ the element $(\ldots, p^{\frac{1}{p}}, p)$ of $\hat\sK^{\flat\sm o}$, $\hat\sK^{\flat\sm o}$ is the $p^\flat$-adic completion of $k[ (p^{\flat})^{\frac{1}{p^\infty}}] $. 
       
  \subsubsection{}    Let $B$ be a $p$-adically complete and $p$-torsionfree $\hat\sK^{\sm o}$-algebra.  Then $B[\frac{1}{p}]$ is canonically a Banach $\sK$-algebra, with unit ball $ p^{-\frac{1}{p^\infty}}B$ . Let   $ B[\frac{1}{p}]^{\sm o}\supset B$ be the ring of power-bounded elements of $B[\frac{1}{p}]$. We say that $B$ is {\it uniform} if $B[\frac{1}{p}]^{\sm o}\subset p^{-N} B$ for some $N\in \N$; any such algebra is {reduced} (because the norm of $B[\frac{1}{p}]$ is then equivalent to the spectral seminorm).
      
    \subsubsection{}  We say that $B$ is {\it perfectoid}\footnote{or: integral perfectoid.} if $F$ induces an isomorphism 
      $B/p^{\frac{1}{p}} \stackrel{\sim}{\to} B/p$.  
      
       In that case $ p^{-\frac{1}{p^\infty}}B = B[\frac{1}{p}]^{\sm o}$, and in particular, $B$ is uniform. In addition, for any $\varpi\in \hat\sK^{\sm o}$ with $\vert p\vert^{\frac{1}{p}} \leq \vert \varpi\vert < 1$, $F$ induces an isomorphism 
      $B/\varpi \stackrel{\sim}{\to} B/\varpi^p$.

   The $p$-adic (separated) completion of any colimit of perfectoid $\hat\sK^{\sm o}$-algebras is perfectoid. 
   
   \begin{prop}\label{P2} Let $\frak a$ be an ideal of $B$ generated by a regular sequence $(p, x_2, \ldots, x_d)$. Then the $\frak a$-adic completion $\hat B^{\frak a}$ is perfectoid.     \end{prop}
   
   \proof Since $\frak a$ is finitely generated, $\hat B^{\frak a}$ is $\frak a$-adically complete; $(p, x_2, \ldots, x_d)$ is a regular sequence in $\hat B^{\frak a}$, and so is $(x_2^n , \ldots, x_d^n, p)$ (\resp $(x_2^n , \ldots, x_d^n, p^{\frac{1}{p}})$)   for every $n>0$. In particular, $\hat B^{\frak a}$ is $p$-torsionfree and $p$-adically complete. Let us show that the ideal $p\hat B^{\frak a}$ (\resp $p^{\frac{1}{p}}\hat B^{\frak a}$)  is closed for the $\frak a$-adic topology, or equivalently, for the $(x_2, \ldots, x_d)$-adic topology. Indeed, if an element $b\in B^{\frak a}$ can be written $ pc_n + d_n$ with $d_n\in (x_2^n , \ldots, x_d^n)B^{\frak a}$ for any $n>0$, then $ pc_n$ is a  $(x_2, \ldots, x_d)$-adic Cauchy sequence. Since $(x_2^n , \ldots, x_d^n, p)$ is a regular sequence, $c_n$ is itself a $(x_2, \ldots, x_d)$-adic Cauchy sequence, and its limit $c$ satisfies $b= pc$. One sees in the same way that  $p^{\frac{1}{p}}\hat B^{\frak a}$   is $\frak a$-adically closed. 
   
     It follows that the map $\hat B^{\frak a}/p \to \widehat{B/p}^{\frak a} $ (\resp $\hat B^{\frak a}/p^{\frac{1}{p}} \to \widehat{B/p^{\frac{1}{p}}}^{\frak a} $) is an isomorphism. On the other hand, denoting by ${\bar{\frak a}}$   the image of $\frak a$ in $B/p$, one has ${\bar{\frak a}}^{pd}\subset  F{\bar{\frak a}}\subset {\bar{\frak a}}^p$, hence $\widehat{B/p}^{\frak a}=\widehat{B/p}^{\bar{\frak a}}=\widehat{B/p}^{F {\bar{\frak a}}} $. Since $B$ is perfectoid, $F$ induces an isomorphism  $\widehat{B/p^{\frac{1}{p}}}^{{\frak a}} = \widehat{B/p^{\frac{1}{p}}}^{\bar{\frak a}} \stackrel{\sim}{\to} \widehat{B/p}^{F{\bar{\frak a}}} = \widehat{B/p}^{ {\frak a}}$, hence also an isomorphism $\hat B^{\frak a}/p^{\frac{1}{p}} \stackrel{\sim}{\to} \hat B^{\frak a}/p $. 
   \qed 
                 
   \subsubsection{}   The {\it tilt} of $B$ is $ B^\flat :=  \lim_{F}\,B/p$. This is a perfectoid $\hat\sK^{\flat\sm o}$-algebra, \ie a $p^\flat$-adically complete, $p^\flat$-torsionfree, perfect $\hat\sK^{\flat\sm o}$-algebra. 
 The natural morphism    \begin{equation}\label{2.1} B^\flat/p^\flat \to B/p  \end{equation}
is an isomorphism. This lifts to a natural isomorphism 
  \begin{equation}\label{2.2} B^\natural :=  W(B^{\flat})\hat\otimes_{W(\hat\sK^{\flat\sm o})} \hat\sK^{\sm o} \stackrel{\sim}{\to} B  \end{equation} 
 (this is well-known for $p^{-\frac{1}{p^\infty}}B =  B[\frac{1}{p}]^{\sm o}$ (Fontaine-Scholze); it follows that \eqref{2.2} is injective in general, and surjectivity is checked by reducing mod. $p$). 

    \subsubsection{}\label{+perf} Let $R$ is a complete local domain of mixed characteristic $(0,p)$, fix a compatible system of roots $p^{\frac{1}{p^i}}$ in $R^+$, and view $R^+$ as a $\sK^{\sm o}$-algebra. Then $\widehat{R^+}$ is perfectoid. 
    
    Indeed, $\widehat{R^+}$ is $p$-torsionfree like $\widehat{R^+}$; since $R^+$ is normal, an equation $x^p= py$ in $R^+$ implies $x/p^{\frac{1}{p}}\in R^+$;  on the other hand, any element of $R^+$ admits $p$-th roots; one concludes that $F$ induces an isomorphism   $R^+/p^{\frac{1}{p}} \stackrel{\sim}{\to} R^+/p$.

   \subsection{Almost perfectoid algebras.}  Let $B$ be a uniform  (complete) $\hat\sK^{\sm o}$-algebra. Let $\pi$ be a non-zero element of $B$ such that a compatible system of roots $\pi^{\frac{1}{p^i}}$ exists (which we fix).    
    \subsubsection{}  We say that $B$ is {\it $\pi^{\frac{1}{p^\infty}}$-almost perfectoid} if it is {\it uniform} and if $F$ induces an almost isomorphism 
      $B/p^{\frac{1}{p}} \stackrel{\sim^a}{\to} B/p$,  \ie kernel and cokernel are killed by  $\pi^{\frac{1}{p^\infty}}$, cf \cite[3.5.4]{A1}.  
       
Its tilt $ B^\flat :=  \lim_{F}\,B/p$ is then a perfectoid $\hat\sK^{\flat\sm o}$-algebra, with a specific element $\pi^\flat =(\ldots, \pi^{\frac{1}{p}}, \pi)$. Moreover  $B^\natural := W(B^{\flat})\hat\otimes_{W(\hat\sK^{\flat\sm o})} \hat\sK^{\sm o}$ is a perfectoid $\hat\sK^{\sm o}$-algebra, and {\it the natural morphism $B^\natural \to B$ is both injective and a $(p\pi)^{\frac{1}{p^\infty}}$-isomorphism } (hence a $\pi^{\frac{1}{p^\infty}}$-isomorphism  if $\pi\in p^{\frac{1}{p^\infty}}B$). Here, in $B^\natural$, $\pi^{\frac{1}{p^i}}$ is identified with the Teichm\"uller lift $[(\pi^{\flat})^{ \frac{1}{p^i}}]$. The formation of $B^\natural$ is functorial in $B$.

Obviously, a uniform $\hat\sK^{\sm o}$-algebra $B$ which is $\pi^{\frac{1}{p^\infty}}$-almost isomorphic to a $\pi^{\frac{1}{p^\infty}}$-almost perfectoid algebra is $\pi^{\frac{1}{p^\infty}}$-almost perfectoid.
 
    \subsubsection{}\label{path} Almost ring theory over the setting $(B,  \pi^{\frac{1}{p^\infty}}B)$ satisfies the general assumption of \cite{GR}: $( \pi^{\frac{1}{p^\infty}}B)^{\otimes 2}$ is flat over $B$ (\cite[2.1.7]{GR}). However, certain ``pathologies" occur when $\pi$ is a zero-divisor in $B$. In order to reduce questions about almost perfectoid algebras to the case when $\pi$ is a non-zero divisor, one may replace $B$ by $\pi^{-\frac{1}{p^\infty}}B$, which is still uniform  \cite[2.5.2]{A1}, hence reduced. The elementary but crucial observation is that {\it in a reduced $\Z[\pi^{\frac{1}{p^\infty}}]$-algebra, any $\pi$-torsion element is $\pi^{\frac{1}{p^\infty}}$-torsion}. Therefore {\it $B\to \pi^{-\frac{1}{p^\infty}}B$ is an almost isomorphism} (so that  $\pi^{-\frac{1}{p^\infty}}B$ is almost perfectoid if $B$ is), and the image of $\pi$ in $\pi^{-\frac{1}{p^\infty}}B$ is a non-zero divisor.  Moreover, if $B$ is $p$-torsionfree, so is $\pi^{-\frac{1}{p^\infty}}B$.

      \subsection{ Cohen-Macaulay algebras.}  Let $R$ be a noetherian local ring of dimension $d$, with maximal ideal $\frak m$. 
        Let $\underline x = (x_1, x_2, \ldots, x_d)$ be a system of parameters, and let $B$ be a (not necessarily noetherian) $R$-algebra.
      
     We say that $\underline x$ becomes regular in $B$, or that $B$ is  {\it Cohen-Macaulay w.r.t. $\underline x$} if for any $i$, 
multiplication by $x_{i+1}$ is injective on $B/(x_1, \ldots, x_i)B$, and $B\neq (x_1, \ldots, x_d)B$. It is a {\it Cohen-Macaulay $R$-algebra} if it is Cohen-Macaulay w.r.t. any system of parameters of $R$.  

\smallskip We shall use freely the following facts: if $B$ is $\frak m$-complete, it is Cohen-Macaulay as soon as it is Cohen-Macaulay w.r.t. some $\underline x\,$; it is then a Cohen-Macaulay $S$-algebra for any factorization $R\to S\to B$ such that $S$ is local and finite over $R$ (since the image of a system of parameters of $R$ in $S$ is a system of parameters, and $B$ is also $\frak m_S$-complete). An algebra over a regular ring $R$ is Cohen-Macaulay if and only it is faithfully flat \cite[1.2.d]{HH2}.  

      \subsection{ Almost Cohen-Macaulay algebras.}  Let $B$ be a uniform (complete) $\hat\sK^{\sm o}$-algebra, and let $\pi^{\frac{1}{p^\infty}}$ be a system of $p^{th}$-power roots of some element $\pi \in B \setminus \{0\}$. 
     
       We say that $\underline x$ becomes $\pi^{\frac{1}{p^\infty}}$-almost regular in $B$, or that $B$ is  {\it $\pi^{\frac{1}{p^\infty}}$-almost Cohen-Macaulay w.r.t. $\underline x$} if for any $i$, 
multiplication by $x_{i+1}$ is almost injective on $B/(x_1, \ldots, x_i)B$, and $B / (x_1, \ldots, x_d)B$ is not almost zero, \ie  $\pi^{\frac{1}{p^\infty}}B\not\subset (x_1, \ldots, x_d)B$. Since $(x_1, \ldots, x_d)R$ contains some power of $\frak m$, the last condition is equivalent to $\pi^{\frac{1}{p^\infty}}B\not\subset \frak m B$, \cf \cite[4.1]{A2}.

\smallskip This terminology may be slightly misleading since ``Cohen-Macaulay" does not formally imply  ``almost Cohen-Macaulay":  $\pi^{\frac{1}{p^\infty}}\not\subset \frak m B$ is stronger than $B\neq \frak m B$. The following proposition allows to settle this issue in some generality.

\begin{prop}\label{P0} Let $R$ be a complete noetherian local domain with maximal ideal $\frak m$, and let $ B$ be a Cohen-Macaulay algebra w.r.t. some system of parameters $\underline x$ of $R$.  Assume that $B$ contains a system of $p^{th}$-power roots $\pi^{\frac{1}{p^\infty}}$ of some element $\pi \in B$. 

\begin{enumerate} \item If $\pi^{\frac{1}{p^\infty}}B \cap R \neq 0$, then $B$  is a $\pi^{\frac{1}{p^\infty}}$-almost Cohen-Macaulay $R$-algebra w.r.t. $\underline x$. 
  \item Let $\frak p$ be a prime ideal of $R$ such that $\pi^{\frac{1}{p^\infty}}B \cap  R \not\subset \frak p R$. 
  Then $ B/(\frak mB +\sqrt{\frak p B}) $ is not $\pi^{\frac{1}{p^\infty}}$-almost zero.
 \end{enumerate} 
\end{prop}  

\proof (1) (\cf also \cite[1.1.2]{A2}) By Cohen's structure theorem, there is a complete regular local subdomain $A$ of $R$ whose maximal ideal is generated by $\underline x$, and $R$ is a finite $A$-module. 
 Since $R$ is a domain and $\pi^{\frac{1}{p^\infty}} B \cap R \neq   0$,  $\pi^{\frac{1}{p^\infty}} B \cap A \neq   0$ as well. On the other hand, $B$ is Cohen-Macaulay over the regular local ring $A$, hence faithfully flat, and in particular  $\frak m_A^n B \cap A = \frak m_A^n$ for every $n$, and their intersection is $0$ by Krull. Since $\frak m_A R$ is $\frak m$-primary, $ \cap\, \frak m^n B \cap A = 0$. If the idempotent ideal  $\pi^{\frac{1}{p^\infty}}B  $ is contained in $\frak m B$, it is contained in $\cap\, \frak m^n B$, a contradiction.  
 
\smallskip (2) Let us set $R'= R/\frak p$, a complete local domain with maximal ideal $\frak m'= \frak m/\frak p$. We cannot apply (1) to the $R'$-algebra $ B/ \sqrt{\frak p B}$, which might not be Cohen-Macaulay; instead, following O. Gabber's suggestion, we will use Hochster's solid closure theory \cite{H2}.  

Let us write $\sqrt{\frak p B}$ as a filtered union of finitely generated ideals $I_\alpha$ of $B$ containing ${\frak p B}$. Since $ B$ is a Cohen-Macaulay algebra w.r.t. $\underline x$ and $(\underline x)R$ is $\frak m$-primary, $B$ has a nonzero $R$-dual \cite[2.4]{H2}. It follows that $B/\frak p B$ has a nonzero $R'$-dual \cite[2.12]{H2}, and further that every $B/I_\alpha$ has a nonzero $R'$-dual \cite[2.1 (m)]{H2}. By \cite[5.10]{H2}\footnote{in the terminology of \cite{H2}, $J (B/I_\alpha) \cap R'$ belongs to the solid closure of $J$.}, it follows that for any given ideal $J$ of $R'$ and every $\alpha$, the intersection $J (B/I_\alpha) \cap R'$  in $B/I_\alpha $ (which equals the intersection $(J B + I_\alpha  ) \cap R'$  in $B$) is contained in the integral closure $\bar J$ of $J$. Their filtered union $J (B/\sqrt{\frak p B}) \cap R'$ is therefore contained in $\bar J$.  

Now $ B/(\frak mB +\sqrt{\frak p B}) = (B/\sqrt{\frak p B})/ \frak m' (B/\sqrt{\frak p B})$, and by the same argument as in (1), it suffices to see that $\cap_n \,\frak m'^n  (B/\sqrt{\frak p B}) \cap R' = 0$. Applying the above observation to $J= \frak m'^n$, this is derived from the classical vanishing of $\,\cap\,  \overline{\frak m'^n}\, $  \cite[5.3.4]{HuS}. \qed

       \section{Getting rid of ``almost": Shimomoto's construction.} 
       \subsection{} Let $R$ be a complete noetherian local domain of mixed characteristic with algebraically closed residue field $k= k^+$ of characteristic $p$. Let $(x_1, x_2, \ldots, x_d)$ be a system of parameters for $R$, with $x_1=p$.
        The following synthetizes, strengthens and recasts according to our needs the results of \cite{Sh2}\footnote{actually in \cite{Sh2}, the $\frak m_R$-adic completion is not performed, and one only gets a Cohen-Macaulay algebra $C$ w.r.t. $(p, x_2, \ldots, x_d)$.} and their proof. 
        
    \begin{thm}\label{T3}  Let $B$ be a $\pi^{\frac{1}{p^\infty}}$-almost $\hat\sK^{\sm o}$-perfectoid almost Cohen-Macaulay $R$-algebra w.r.t. $\underline x\,$ for some 
  {$\pi \in p^{\frac{1}{p^\infty}}B\,$} (it is tacitly assumed that the $W(k)$-structures coming from $R$ and $\hat\sK^{\sm o}$ coincide).
    \begin{enumerate}   
   \item 
   There exists a $\hat\sK^{\sm o}$-perfectoid Cohen-Macaulay $R$-algebra $C$ (which one may assume to be $\frak m_R$-adically complete) and a morphism $B^{\natural}\to C$.  
   \item One may assume moreover that $C$ is a $R^+$-algebra. 
    \end{enumerate} \end{thm}
   
  The existence of $B^{\natural}\to C$ will be crucial in the proof of Theorem \ref{T2} (caution: $B^{\natural}$ may not be an $R$-algebra).
 In its standard version, Hochster's  modification process to construct a Cohen-Macaulay $R$-algebra $C$, as applied for instance in \cite{A2}, starts with an almost Cohen-Macaulay $B$ as an auxiliary object, without relating $B$ and $C$ explicitly (the almost Cohen-Macaulay condition is there to ensure that the process does not degenerate). Dietz's work remedies this shortfall by providing a map $B \to C$ in char. $p$, while Shimomoto ``transposes" it into mixed characteristic by tilting\footnote{after the release of the first version of this paper, I learned that O. Gabber has a alternative way of passing from almost perfectoid almost Cohen-Macaulay algebras to perfectoid  Cohen-Macaulay algebras, using an ultraproduct technique.}.
 
 For convenience, we restate the fragment from Dietz's theory that we need in the following form:
 
 \begin{prop}\label{P3} Let $S$ be a complete noetherian local $k$-algebra with residue field $k$. \begin{enumerate}   
   \item  Any $(\pi^\flat)^{\frac{1}{p^\infty}}$-almost Cohen-Macaulay $S$-algebra $D$ maps to a Cohen-Macaulay $S$-algebra which is a perfect domain (\resp $\hat\sK^{\flat \sm o}$-perfectoid, if $p^\flat$ is a parameter of $S$).
   \item Let $(D_\alpha)$ be a directed system of $(\pi_\alpha^\flat)^{\frac{1}{p^\infty}}$-almost Cohen-Macaulay $S$-algebras (for some $\pi_\alpha^\flat\in D_\alpha \setminus 0$). Then ${\rm{colim}} \, D_\alpha$ maps to a Cohen-Macaulay $S$-algebra which is a perfect domain (\resp $\hat\sK^{\flat \sm o}$-perfectoid, if $p^\flat$ is a parameter of $S$).  \end{enumerate}
 \end{prop}
 
     \begin{proof} $(1)$ In the terminology of \cite{D}, the fact that $D$ is $(\pi^\flat)^{\frac{1}{p^\infty}}$-almost Cohen-Macaulay translates into: $\pi^\flat$ is a ``durable colon-killer" over $S$; thus $D$ is a ``seed" by  \cite[4.8]{D}, and any seed maps to a Cohen-Macaulay algebra which is a perfect and $\frak m_S$-separated by  \cite[3.7]{D} and is also a domain by \cite[7.8]{D}) - hence after $p^\flat$-adic completion, to a  $\hat\sK^{\flat \sm o}$-perfectoid Cohen-Macaulay $S$-algebra if $p^\flat$ is a parameter of $S$. 
    
    For $(2)$, one observes that a directed colimit of seeds is a seed \cite[3.2]{D}.
        \end{proof}
     
 \begin{proof} (of Theorem \ref{T3}) 

$(1)$  
 Up to replacing $B$ by $B\langle \underline x^{\frac{1}{p^\infty}}\rangle[\frac{1}{p}]^{\sm o}$, we may also assume that {\it $B$ contains a system of $p^{th}$-power roots $x_i^{\frac{1}{p^\infty}}$ of each $x_i$} (indeed, $B\langle \underline x^{\frac{1}{p^\infty}}\rangle[\frac{1}{p}]^{\sm o}/p$ is $p^{\frac{1}{p^\infty}}$-almost faithfully flat over $B/p$ \cite[2.5.2]{A2}, hence $B\langle \underline x^{\frac{1}{p^\infty}}\rangle[\frac{1}{p}]^{\sm o}$ is, like $B$, $\pi^{\frac{1}{p^\infty}}$-almost Cohen-Macaulay w.r.t. $\underline x$). We may also replace $B$ by the $\pi^{\frac{1}{p^\infty}}$-almost isomorphic algebra   $\pi^{-\frac{1}{p^\infty}}B$, hence assume that {\it $\pi$ is not a zero divisor in $B$}, and that $B = \pi^{-\frac{1}{p^\infty}}B^\natural$.

\smallskip Let us first assume that {\it $R$ is normal}. 
  Let $\tilde R := R[\frac{1}{p}]^{et, \sm o}$ be the integral closure of $R$ in the maximal etale extension of $R[\frac{1}{p}]$  in some fixed algebraic closure $K^+$ of the field of fractions of $R$. Let us write $\tilde R= {{\rm{colim}}}\, R_\alpha$ as a directed colimit of finite normal extensions of $R$ such that $R_\alpha[\frac{1}{p}]$ is etale over $R[\frac{1}{p}]$. The $p$-adic completion $\hat{\tilde{R}} := \widehat{{\rm{colim}}}\, R_\alpha$ of $\tilde R$ is $\hat\sK^{\sm o}$-perfectoid, \cf \eg \cite[10.1]{Sh1}. 
  
  Let $\tilde B $ be the integral closure of $B $ in $\tilde R\otimes_R B[\frac{1}{\pi}]$.
   Its completion {\it $\hat{\tilde{B}}$ is a $\pi^{\frac{1}{p^\infty}}$-almost perfectoid $\hat{\tilde{R}}$-algebra, and  $\hat{\tilde{B}} /p$ is $\pi^{\frac{1}{p^\infty}}$-almost faithfully flat over $B/p$}. 
  
  Indeed,    $\hat{\tilde{B}}  $ is the completed colimit of 
    the integral closures $ B_{(\alpha)}$ of $B = \pi^{-\frac{1}{p^\infty}} B^\natural  $ in $ R_\alpha\otimes_R B[\frac{1}{\pi}]$. Since  $ R_\alpha\otimes_R B[\frac{1}{\pi}]$ is finite etale over $ B[\frac{1}{\pi}] = B^\natural[\frac{1}{\pi}]$, the perfectoid Abhyankar lemma \cite[0.3.1, 5.3.1]{A1}\footnote{see also a very short account of its proof in \cite{A3}. In the notation of \cite{A1}, $B^\natural[\frac{1}{p}]$ is $\sA$, $ R_\alpha\otimes_R B[\frac{1}{\pi}]$ is $\sB'$ and $ B_{(\alpha)}$ is $\sB^{\sm o}$.} applies to $(B^\natural[\frac{1}{p}], \,  R_\alpha\otimes_R B[\frac{1}{\pi}])$ and shows that $ B_{(\alpha)}$ is uniform, $(B_{(\alpha)})^\natural$ is perfectoid, $ B_{(\alpha)}$ is $\pi^{\frac{1}{p^\infty}}$-almost perfectoid, and 
  $B_{(\alpha)}/p$ is $\pi^{\frac{1}{p^\infty}}$-almost faithfully flat over $B/p$.  
  
  Since $\underline x$ is a $\pi^{\frac{1}{p^\infty}}$-almost regular sequence in $B$, we deduce that the image of $(x_2, \ldots, x_d)$ is also a  $\pi^{\frac{1}{p^\infty}}$-almost regular sequence in $\hat{\tilde{B}}/p$.  Tilting $\hat{\tilde B}$, we get a perfectoid $\hat\sK^{\flat \sm o}$-algebra $\hat{\tilde B}^\flat$ and a $\pi^{\flat\frac{1}{p^\infty}}$-almost regular sequence $\underline x^\flat = (p^\flat, x_2^\flat, \ldots, x_d^\flat) $, where $x_i^\flat = (\ldots, x_i^{\frac{1}{p}}, x_i)$ (note that the $\pi^{\flat\frac{1}{p^\infty}}$-almost isomorphism $\hat{\tilde{B}}^\flat/p^\flat\to  \hat{\tilde{B}}/p$ sends $\underline x^\flat$ to $\underline x$).   
   
  By Proposition \ref{P3} $(1)$, one can map $\hat{\tilde{B}}^\flat$ to a $\hat\sK^{\flat \sm o}$-perfect(oid) Cohen-Macaulay $k[[p^\flat, x_2^\flat, \ldots, x_d^\flat]]$-algebra.
    Untilting, we get a $\hat\sK^{ \sm o}$-perfectoid Cohen-Macaulay algebra w.r.t. $\underline x$. Completing $(p,x_2,\ldots, x_d)$-adically, we get a $\hat\sK^{ \sm o}$-perfectoid Cohen-Macaulay $R$-algebra $C$ (Proposition \ref{P2}), and morphisms $B^\natural \to \hat{\tilde B}^\natural\to C.   $
     
  \smallskip We next drop the assumption that $R$ is normal\footnote{we could actually dispense with this step which is subsumed in the proof of $(2)$ below.}.  Let $R^{\rm n}$ be the normalization of $R$ (which is again a complete noetherian local domain), and $g\in R$ be such that $R^{\rm n}[\frac{1}{pg}]$ is etale over $R[\frac{1}{pg}]$. We may (by the same argument as above) assume that $g$ is a non-zero divisor in $B$ and that $B$ contains $g^{\frac{1}{p^\infty}}$.  
   Let $\tilde B^{\rm n}$ be the integral closure of $ B $ in $(\tilde R^{\rm n}\otimes_R B)[\frac{1}{\pi g}]$.  We then construct  $C$ and morphisms $B^\natural  \to \hat{\tilde B}^{{\rm n}\natural}\to C$ as above, on replacing $(\tilde R, \tilde B , \pi )$ by $(\tilde R^{\rm n}, \tilde B^{\rm n}, \pi g)$.

   \medskip   $(2)$  Let us write $R^+$ as a directed colimit of finite normal $R$-subalgebras $R_\beta$, which are complete local domains with residue field $k$. Setting $\tilde R_\beta := R_\beta[\frac{1}{p}]^{et, \sm o}$ as above, we also have $R^+ = {\rm{colim}} \, \tilde R_\beta $. Recall that, once a compatible system of roots $p^{\frac{1}{p^i}}$ is chosen in $R^+$, the $p$-adic completion $\widehat{R^+}$ is perfectoid over $\hat\sK^{\sm o}$ (\ref{+perf}) .

  Let $\underline g_\beta = ( g_{\beta, 1}, g_{\beta, 2}, \ldots )$ be a finite sequence of elements of  $R_\beta$ such that $R_\beta[\frac{1}{p \Pi g_{\beta, i}}] $ is etale over $R[\frac{1}{p \Pi g_{\beta i}}] $. We consider the directed system formed by pairs $\underline\beta = (\beta,  \underline g_\beta)$ (with the order $(\beta,  \underline g_\beta)\leq (\beta',  \underline g_{\beta'})$ is $\beta\leq \beta'$  and $\underline g_\beta$ is part of the sequence $\underline g_{\beta'}$). 
  
  We define $B_{\underline\beta} :=  (\Pi g_{\beta i})^{-\frac{1}{p^\infty}} B\langle   \underline g_\beta^{ \frac{1}{p^\infty}}  \rangle[\frac{1}{p}]^{\sm o}$, which is $(\pi{\underline g_\beta})^{\frac{1}{p^\infty}}$-almost perfectoid. When ${\underline\beta}$ varies, they form a directed system of $R $-algebras. By \cite[2.5.2]{A2}, $B_{\underline\beta}/p$ is $(\pi{\underline g_\beta})^{\frac{1}{p^\infty}}$-almost faithfully flat over $B/p$.  
  
    We further define $\tilde B_{\underline\beta}$ to be the integral closure of $  B_{\underline\beta}\otimes_R \tilde R_\beta$ in  $  (B_{\underline\beta}\otimes_R \tilde R_\beta) [\frac{1}{p \Pi g_{\beta, i}}]$. They form a directed system of $ \tilde R_\beta$-algebras. Since
  $ B_{\underline\beta}$ contains a system of roots $(\Pi g_{\beta, i})^{\frac{1}{p^\infty}}$ of  the non-zero divisor $\Pi g_{\beta, i} \in  B_{\underline\beta}$, the perfectoid Abhyankar lemma shows as above that $ \hat{\tilde B}_{\underline\beta}$ is $(\pi{\underline g_\beta})^{\frac{1}{p^\infty}}$-almost perfectoid, and that $(\tilde B_{\underline\beta})/p$ is $(\pi{\underline g_\beta})^{\frac{1}{p^\infty}}$-almost faithfully flat over $B_{\underline\beta}/p$, hence over $B/p$, so that $ \hat{\tilde B}_{\underline\beta}$ is $(\pi{\underline g_\beta})^{\frac{1}{p^\infty}}$-almost Cohen-Macaulay w.r.t. \underline x. We may then imitate the argument in $(1)$: the $\hat{\tilde R}_\beta^{\flat}$-algebras $ \hat{\tilde B}_{\underline\beta}^\flat$ form a directed system of $(\pi{\underline g_\beta})^{\frac{1}{p^\infty}}$-almost Cohen-Macaulay $k[[\underline x^\flat]]$-algebras, hence their $p^\flat$-adically completed colimit  $\,  \widehat{\rm{colim}}_{\underline\beta} \, \hat{\tilde B}_{\underline\beta}^\flat\cong ( \widehat{\rm{colim}}_{\underline\beta} \, \hat{\tilde B}_{\underline\beta}^\natural)^\flat \,$ 
   maps  to a $\hat\sK^{\flat \sm o}$-perfect(oid) Cohen-Macaulay $k[[p^\flat, x_2^\flat, \ldots, x_d^\flat]]$-algebra by Proposition \ref{P3} $(2)$.
      Untilting and completing $(p,x_2,\ldots, x_d)$-adically, we get the perfectoid Cohen-Macaulay $R$-algebra $C$ (which is an algebra over 
     $\widehat{R^+} = \widehat{\rm{colim}} \,  \hat{\tilde R}_\beta$), and we also get a morphism $B^\natural \to \widehat{\rm{colim}} \, \hat{\tilde B}_{\underline\beta}^\natural \to C$.   \end{proof}
   
    \subsection{} In \cite[\S 4]{A2}, the existence of a Cohen-Macaulay $R$-algebra is shown by first constructing a $\pi^{\frac{1}{p^\infty}}$-almost perfectoid $\pi^{\frac{1}{p^\infty}}$-almost Cohen-Macaulay $R$-algebra $B\,$ (where $\pi\in pR$, $R[\frac{1}{\pi}]$ is etale over $W(k)[[x_2, \ldots, x_d]][\frac{1}{\pi}]$, and $B$ is $\pi$-torsionfree), and then applying the technique of algebra modifications to $B$. Using Theorem \ref{T3} instead of the latter, one gets the following more precise result:
 
 \begin{thm}\label{T4}\cite{Sh2}\footnote{see footnote 4.} For any complete noetherian local domain $R$ of mixed characteristic with perfect residue field $k$, there exists a $\hat\sK^{\sm o}$-perfectoid $\frak m_R$-adically complete Cohen-Macaulay $R$-algebra $C$. Moreover, one can assume that $C$ is an $R^+$-algebra (in such a way that the $\sK^{\sm o}$-structures coming from $R^+$ and $\hat\sK^{\sm o}$ coincide).     \qed \end{thm}

     \section{Proof of Theorem \ref{T2}.}

\subsection{Reduction of Theorem \ref{T2} to the case of a surjection, and a stronger statement in that case.}   
   
       \subsubsection{}   First, we can reduce to the case when $k'$ is {algebraically closed}. Indeed, it suffices to replace $R'$ by an extension $R''$ defined as follows: 
       
           $i)$ if $R'$ is of mixed characteristic, let $\Lambda$ be a coefficient ring and $(x'_1= p, x'_2, \ldots, x'_d)$ be a system of parameters; then $R'$ is a finite extension of $\Lambda[[x'_2, \ldots, x'_{d'}]]$ by Cohen's theorem, and it follows that $R'\hat\otimes_\Lambda W(k'^+)\cong R'\otimes_{\Lambda[[x'_2, \ldots, x'_{d'}]]}W(k'^+)[[x'_2, \ldots, x'_{d'}]]$ is a complete noetherian local $R'$-algebra of dimension $d'$; we take for $R''$ the quotient of $R'\hat\otimes_\Lambda W(k'^+)$ by some minimal prime with $\dim R'' = d'$; 
           
            $ii)$ if $R'$ is of characteristic $p$, we make a similar construction replacing $\Lambda[[x'_2, \ldots, x'_{d'}]]$ by $k[[x'_1, x'_2, \ldots, x'_{d'}]]$.

\medskip     According to \cite[1.1]{AFH}, one has a ``Cohen factorization"  $R \to R^\cdot \to R''= R^\cdot/\frak p$ where $R\to R^\cdot$ is flat. Since every system of parameters for $R$ extends to a system of parameters for $R^\cdot$ by flatness, any Cohen-Macaulay $R^\cdot$-algebra is also a Cohen-Macaulay $R$-algebra. It thus suffices to treat the case of $R^\cdot \to R''$; in other terms, we may assume from the beginning that {\it $R\to R'$ is surjective} - and in particular $R$ and $R'$ {\it have the same (algebraically closed) residue field $k= k'$}. 

  \smallskip   These reductions are compatible with the prescription of $f^+$ in Theorem \ref{T1}, since we may replace at once $R$ by its normalization and apply Proposition \ref{P1} $(3)$ to the sequences $R \to R' \to R''$ and  $R\to R^\cdot    \to R''$.

    \subsubsection{}   This discussion shows that Theorem \ref{T1} follows from:
   
     \begin{thm}\label{T5} Let $R, C$ be as in Theorem \ref{T4}. 
     
     Let ${f}: R\to R':= R/\frak p$ be the quotient map by a prime ideal. Then there exists $\frak m_{R'}$-adically complete Cohen-Macaulay $R'$-algebra $C'$, which is $\hat \sK^{\sm o}$-perfectoid if $R'  $ is of mixed characteristic (\resp perfect if $R'$ is of characteristic $p$), and a commutative diagram
       \begin{equation}\label{CC3}   \xymatrix @-1pc {   R  \ar[d]   \,  \ar@{->}[r]^{{f}}    & R'  \ar[d]  \\  \;\, R^+  \ar[d]   \,  \ar@{->}[r]^{{f}^+}    & \;\, R'^+  \ar[d]    \\  C  \,  \ar@{->}[r]   & C'     .  }  \end{equation} 
  Moreover, ${f}^+$ can be given in advance.   
        \end{thm}

 The fact that $C$ is given ``in advance" is crucial, as it permits to compose such diagrams for a chain of quotients $R\to R/\frak p_1 \to \cdots \to R/\frak p_n$, and conversely to decompose $R\to R'$ into successive quotients by primes of height $1$  (taking into account Proposition \ref{P1}).  
   
    We thus may and shall assume that {\it $\frak p$ is of height $1$}. Replacing $R$ by its normalization (and $\frak p$ by a prime over it), we may and shall assume in addition that {\it $R$ is normal}.

   \subsection{Proof of Theorem \ref{T5}: case when $R'$ is of mixed characteristic.}  If $R'$ is of mixed characteristic,    
     $R_{\frak p}$ is a discrete valuation ring of equal characteristic $0$.  We proceed in five steps.

     \subsubsection{Choosing an appropriate system of parameters of $R$.}\label{seq}    $\,$ 
      
      We take $x_1 = p$. 
      
      We choose $x_2  \in \frak p$ such that $x_2$ generates $\frak p R_{\frak p}$ and does not lie in any minimal prime ideal $\frak q$ of $R$ above $p$.  
      
  Such an element exists. Indeed, since $R/\frak p$ is of mixed characteristic, $\frak p$ is not contained in any (height one prime) 
   $\frak q$. By prime avoidance there exists $y\in \frak p$ such that $y$ does not belong to any $\frak q$. If $y$ generates $\frak p R_{\frak p}$, we take $x_2= y$. Otherwise, let $x \in \frak p$ generate $\frak p R_{\frak p}$; by the affine version of prime avoidance\footnote{due to E. Davis, \cf \cite[th. 124]{K} (I am indebted to P. Roberts for this reference). We recall the proof for convenience. Let us assume on the contrary that $x+ yR \subset \cup_1^n \frak q_i$ (with $n$  minimal), and select $m$ such that $x\in \frak q_i, i\leq m \leq n$ but $ x\notin \cup_{m+1}^n \frak q_i$. For every $j>m$, let $z_j\in \frak q_j$ be such that $ z_j \notin \frak q_\ell, \, \ell\neq j$. If $Ry\subset  \cup_1^m  \frak q_i$, then $(x, y)R\subset \frak p_i$ for some $i\leq m$ by prime avoidance: a contradiction. 
 Therefore there exists $y'\in  Ry$ not in $\cup_1^m   \frak q_i$, and then $x_2  = x + y'\Pi_{m+1}^n z_j\notin \cup_1^n \frak q_i$.
 
 Actually, since $C$ is reduced, it is not essential for the sequel to assume that $x_2$ generates $\frak p R_{\frak p}$.},
  there exists $x_2\in x + yR$ which does not belong to any $\frak q$. 
  
   The pair $(p  , x_2)$ is then part of a system of parameters $\underline x = (x_1, x_2, \ldots, x_d)$ of $R$, which we fix.  
 We also fix a compatible sequence of roots $x_2^{\frac{1}{e}}  \subset R^+$, and still denote by the same symbols their images in $C$. We denote by $(x_2^{\frac{1}{\infty}})$ the ideal $\cup  x_2^{\frac{1}{e}}C$ and by $(x_2^{\frac{1}{\infty}})^-$ its closure.
        
    \subsubsection{$\bar C :=  C / (x_2^{\frac{1}{\infty}})^-$ is perfectoid,  and $\underline x' := (p, x_3, \ldots, x_d)$ becomes a regular sequence in $\bar C$.}\label{per}  $\,$
    
    For any $e\in \N\setminus 0$, $R[x_2^{\frac{1}{e}}]$  is a finite extension of the regular ring $  W(k)[[x_2^{\frac{1}{e}}, x_3, \ldots, x_d]]$.  On the other hand, $C$ is  a Cohen-Macaulay $R[x_2^{\frac{1}{e}}]$-algebra, hence faithfully flat over $W(k)[[x_2^{\frac{1}{e}}, x_3, \ldots, x_d]]$, hence also over their directed colimit $   W(k)[[x_2^{\frac{1}{\infty}}, x_3, \ldots, x_d]]$. 
     Thus $C/(x_2^{\frac{1}{\infty}})$ is faithfully flat on $A' := W(k)[[ x_3, \ldots, x_d]],$ and so is its (separated) completion $\bar C$ \cite[1.1.1]{A2}. Therefore $\bar C$ is a Cohen-Macaulay $A'$-algebra.

\smallskip On the other hand, $C$ is $p$-torsionfree, and $x_2^{\frac{1}{e}}$ is a non-zero divisor in $C/p$, so that $x_2^{\frac{1}{e}}$ acts isometrically on $C$ (and $x_2^{\frac{1}{e}} (C/p ) = (x_2^{\frac{1}{e}} C)/p $). This implies that for any $j>0$,  the kernel of $C/ (p^{j }, (x_2^{\frac{1}{ \infty}})) \stackrel{\times p}{\to} C/ (p^{j }, (x_2^{\frac{1}{ \infty}})) $ is $p^{j-1}(C/ (p^{j }, (x_2^{\frac{1}{ \infty}}))$. Passing to the limit on $j$, we see  that  $ \bar C  $ is $p$-torsionfree. In turn, $C$ being $\hat\sK^{\sm o}$-perfectoid, Frobenius  induces isomorphisms 
          \begin{equation}     (x_2^{\frac{1}{pe} } C)/p^\frac{1}{p}  = x_2^{\frac{1}{pe}}  (C/p^\frac{1}{p} ) \stackrel{\sim}{\to} x_2^{\frac{1}{e}} (C/p ) = (x_2^{\frac{1}{e}} C)/p   .\end{equation} Passing to the colimit on the integers $e$ (ordered by divisibility), we conclude  that
 $\bar C $ is $\hat\sK^{\sm o}$-perfectoid. 
    
  \subsubsection{Constructing an almost perfectoid almost Cohen-Macaulay $R'$-algebra $B'$ w.r.t. $\underline x'$.}\label{alm}

Since $R_{\frak p}$ is a discrete valuation ring with uniformizer $x_2$, there exists $\pi\in R\setminus \frak p$ such that 
   \begin{equation}\label{4.2} \pi\frak p  \subset x_2R.\end{equation}  Since $R'= R/\frak p$ is $p$-torsionfree, we may also assume that $p$ divides $\pi$.
 
  We fix a compatible system of roots $\pi^{\frac{1}{p^i}}$ in $R^+$ and view $C$ and $\bar C$ as $\Z[\pi^{\frac{1}{p^\infty}}]$-algebras.
       Since $\bar C$ is perfectoid, hence reduced, \eqref{4.2} implies that 
        \begin{equation}\label{4.3} 
 \pi^{{\frac{1}{p^\infty}}} \frak p C \subset  \sqrt{x_2 C}   \subset (x_2^{\frac{1}{\infty}})^-.\end{equation} 
 In particular $\bar C/ p$ is $\pi^{{\frac{1}{p^\infty}}}$-almost isomorphic to $C/ (pC + \sqrt{\frak p C})$. Since $\pi \in R\setminus \frak p$, Proposition \ref{P0} (2) shows that $C/ (\frak m C + \sqrt{\frak p C})$ is not almost zero, and we conclude that $\bar C$ is a $\pi^{{\frac{1}{p^\infty}}}$-almost Cohen-Macaulay $R/x_2$-algebra w.r.t. $\underline x'$. 

However, $\bar C$ may have $\pi^{{\frac{1}{p^\infty}}}$-torsion, hence not be a $R'$-algebra.   We introduce
   $$B' := \pi^{-{\frac{1}{p^\infty}}} \bar C, $$
  a sub-$\sK^{\sm o}$-algebra of $\bar C[\frac{1}{\pi}]$. It is $p$-torsionfree and $p$-adically complete, and also uniform \cite[2.5]{A1}. In fact, it is $\pi^{\frac{1}{p^\infty}}$-almost perfectoid and  $\pi^{{\frac{1}{p^\infty}}}$-almost Cohen-Macaulay w.r.t.  
  $ \underline x'   $ since $\bar C \to B'$ is a $\pi^{\frac{1}{p^\infty}}$-almost isomorphism, \cf \ref{path}. By \eqref{4.2}, $\frak p$ goes to $\pi$-torsion in $B'$ (hence to $0$), so that $R\to B'$ factors through $R'  $.
  
    \subsubsection{Obtaining $C'$ using Shimomoto's construction.}\label{fin} We may apply Theorem \ref{T3} to $B'$, which is a $\pi^{\frac{1}{p^\infty}}$-almost $\hat \sK^{\sm o}$-perfectoid, almost Cohen-Macaulay $R'$-algebra w.r.t.  $\underline x'  $. This provides a $\hat\sK^{\sm o}$-perfectoid $\frak m_{R'}$-adically complete Cohen-Macaulay $R'$-algebra $C'$ (and we may assume that $C'$ is an $R'^+$-algebra), together with a morphism of perfectoid algebras $ B'^\natural \to  \hat{\tilde B}'^{{\rm n} \natural } \to C'$.   Moreover,  $ \hat{\tilde B}'^{{\rm n} \natural }$ is an $ \hat{\tilde R}'^{\rm n}$-algebra.      On the other hand, since $\bar C$ is perfectoid, the composed $R$-morphism $\bar C\to B' \to \hat{\tilde B}'^{{\rm n} }$ factors through $\hat{\tilde B}'^{{\rm n} \natural }\subset \hat{\tilde B}'^{{\rm n}  }$, and we conclude that the composed morphism $C\to \bar C \to    \hat{\tilde B}'^{{\rm n} \natural } \to C'$ is compatible with $f$.  
  
 \smallskip  This proves Theorem \ref{T5} except for the factorization through ${f}^+$ (the kernel of $R^+ \to C'$ might not even be a prime above $\frak p$). But this is already {\it enough to prove Theorem \ref{T1} in the mixed characteristic case}.

   \subsubsection{Taking care of ${f}^+$.} We now fix a lifting ${f}^+: R^+\to R'^+$ of ${f}$ (which exists by Proposition \ref{P1}), and denote by $\frak p^+$ its kernel, which is a prime of $R^+$ above $\frak p$.  In order to take ${f}^+$ into account, we will have to modify $C'$ in the spirit of part $(2)$ of Theorem \ref{T3}.
 By the (usual) Abhyankar lemma, we may write $R^+$ as the directed colimit of normal finite sub-$R$-algebras $R_\beta$ such that for some $e_\beta\in \N$,  $R_\beta$ contains $x_2^{\frac{1}{e_\beta}}$ and its ramification index is $e_\beta$ at $\frak p_\beta := \frak p^+\cap R_\beta$ (the latter condition amounts to: $\, R_\beta\subset K'[[x_2^{\frac{1}{e_\beta}}]]$). We note that $R_\beta$ is also a complete local domain with residue fied $k$.

We set $R'_\beta := R_\beta/\frak p_\beta$. Then $R'^+= R^+/\frak p^+ = \cup R'_\beta$, and  $ R^+\stackrel{f^+}{\surj} R'^+$ induces maps $ R_\beta\stackrel{f_\beta}{\surj} R'_\beta$ and coincides with their directed colimit.
  We can perform the previous steps with $(R_\beta, \frak p_\beta)$ instead of  $(R, \frak p)$. In step  \ref{seq}, we can replace $\underline x$ by the system of parameters $\underline x_\beta := (p, x^{\frac{1}{e_\beta}}, x_3, \ldots, x_d)$.  In step  \ref{per}, the same $\bar C$ works for every $\beta$. In step  \ref{alm}, we can choose an element $\pi_\beta \in  R_\beta \setminus \frak p_\beta $ divisible by $p $, such that $\pi_\beta  \frak p_\beta \subset x^{\frac{1}{e_\beta}}R_\beta$. In this way, the pairs $\underline \beta  = (\beta,   \pi_\beta  )$ form a directed set (with the order $(\beta,   \pi_\beta) \leq (\beta',   \pi_{\beta'})$ if $\beta\leq \beta'$ and $ \pi_\beta$ divides $ \pi_{\beta'} $). 

 We define the $R'_\beta$-algebra $B'_{\underline\beta} := {\pi_\beta}^{-\frac{1}{p^\infty}}\bar C$, and the $\tilde R'_\beta$-algebra $\tilde B'_{\underline\beta} $ as in part $(1)$ of the proof of Theorem \ref{T3}. 
 The tilts $\widehat{\tilde B'_{\underline\beta}}^\flat $ form a directed system of $\widehat{R'_{ \beta}}^\flat$-algebras, which are ${(\pi_\beta^\flat)}^{\frac{1}{p^\infty}}$-almost Cohen-Macaulay $k[[p^\flat, x_3^\flat, \ldots, x_d^\flat]]$-algebras. Their ($p^\flat$-adically completed) colimit w.r.t. $\underline\beta$ maps to a $\hat\sK^{\flat \sm o}$-perfect(oid) Cohen-Macaulay $k[[p^\flat, x_3^\flat, \ldots, x_d^\flat]]$-algebra by Proposition \ref{P3} $(2)$.
   Untilting and completing $\frak m_{R'}$-adically, we get the perfectoid Cohen-Macaulay $R'$-algebra $C'$ (Proposition \ref{P2}), which is also an algebra over 
     $\widehat{R'^+} = \widehat{\rm{colim}} \, \widehat{R'_{ \beta}}$, and get a composed morphism $C \to  \bar C \to \widehat{\rm{colim}} \, \hat{\tilde B}_{\underline\beta}^\natural \to C'$ which is compatible with 
 ${f}^+ = {\rm{colim}}\,f_\beta$.   \qed
  
     \subsection{Proof of Theorem \ref{T5}: case when $R'$ is of characteristic $p$.}      The proof is similar, somewhat simpler.  We proceed in four steps. 
  
   \subsubsection{Choosing  an appropriate system of parameters of $R$.} By the Cohen-Gabber structure theorem \cite[IV 4.2]{ILO} (based on Epp's elimination of wild ramification), there exists a finite extension $S$ of $R$ in $R^+$ which becomes a finite extension of a ring $A := V[[x_2, \ldots, x_d]]$ (where $V$ is a complete discrete valuation ring with residue field $k = k^+$), such that $A\to S$ becomes etale after inverting some element of $A$ not contained in any minimal prime above $p$. Let $x_1$ be a uniformizer of $V$. One has $V[\frac{1}{p}]=  V[\frac{1}{x_1}]$, and if $e$ denotes the absolute ramification index of $V$, $p^{-\frac{1}{e}}x_1$ is a unit in the discrete valuation ring $V[p^{\frac{1}{e}}]$.
   
 We may and shall replace $R$ by $S$ (and $\frak p$ by a prime above it).  Then  $\underline x := (x_1, \ldots, x_d)$ is a system of parameters of $R$, $x_1$ is a uniformizer of the discrete valuation ring of mixed characteristic $R_{\frak p}$,  $R[\frac{1}{p}]=  R[\frac{1}{x_1}]$, $A\cap \frak p = x_1A$, and the image $\underline x'$ of $(x_2, \ldots, x_d)$ in $R'$ is a system of parameters. 
 We choose a compatible sequence of roots $x_1^{\frac{1}{p^i}}$ in $C$.  
   
   \subsubsection{$\bar C :=  C/x_1^{\frac{1}{p^\infty}}C\,$ is a perfect Cohen-Macaulay $R/x_1$-algebra w.r.t. $\underline x'$.} Since $C$ is actually perfectoid over the perfectoid ring $\hat\sK^{\sm o}[p^{\frac{1}{e}}]= \widehat{W[p^{\frac{1}{p^{e\infty}}}]}$, $F$ induces an isomorphism  
      $C/p^{\frac{1}{ep^{i+1}}} \stackrel{\sim}{\to} C/p^{\frac{1}{ep^{i}}}$, hence an isomorphism 
      $C/x_1^{\frac{1}{p^{i+1}}} \stackrel{\sim}{\to} C/x_1^{\frac{1}{p^{i}}}$ since $p^{-\frac{1}{e}}x_1$ is a unit in $C$.   Passing to the colimit on $i$, we see that $\bar C$ is perfect. 
      
    As in \ref{per}, for any $i$, $C $ is faithfully flat over    $V[x_1^{\frac{1}{p^i}}][[x_2, \ldots, x_d]]$, hence also over their directed colimit. Thus $\bar C$ is faithfully flat over $A':=  k[[ x_2, \ldots, x_d]]$.  
 
    \subsubsection{Constructing an almost Cohen-Macaulay $R'$-algebra $B'$ w.r.t. $\underline x'$.} Since $R_{\frak p}$ is a discrete valuation ring with uniformizer $x_1$, there exists $\pi\in R\setminus \frak p$ such that 
   \begin{equation}\label{4.4} \pi\frak p  \subset x_1R.\end{equation}  
  We fix a compatible system of roots $\pi^{\frac{1}{p^i}}$ in $R^+$ and view $C$ and $\bar C$ as $\Z[\pi^{\frac{1}{p^\infty}}]$-algebras.
        Since $\bar C$ is perfect, hence reduced, \eqref{4.4} implies that 
        \begin{equation}\label{4.5} 
 \pi^{{\frac{1}{p^\infty}}} \frak p C\subset  \sqrt{x_1 C}   \subset x_1^{\frac{1}{p^\infty}}C.\end{equation} 
     By the same argument as above, using Proposition \ref{P0}, $\bar C$ is $ \pi^{{\frac{1}{p^\infty}}}$-almost Cohen-Macaulay w.r.t. $\underline x'$. It may have $\pi^{{\frac{1}{p^\infty}}}$-torsion, and as above, we introduce  $ B' := \pi^{-{\frac{1}{p^\infty}}} \bar C $, which is still $\pi^{{\frac{1}{p^\infty}}}$-almost Cohen-Macaulay w.r.t. $\underline x'$.  
     
     \subsubsection{Obtaining $C'$ using Dietz's construction.} By Proposition \ref{P3} $(1)$,
       $B'$ maps to a perfect, $\frak m_{R'}$-adically separated, Cohen-Macaulay
 $R'$-algebra without zero-divisor. Denoting by $C'$ its $\frak m_{R'}$-adic completion, we have a homomorphism $C\to C'$ compatible with $f$. The kernel of $R^+ \to C'$ is a prime $\frak p^+$ lying above $\frak p$, so that $C\to C'$ factors through some lifting $ R^+ \to R'^+$ of $f$.  
 
 Moreover, since  $R$ is normal (after preliminary reduction), any lift $f^+$ of $f$  equals the previous one precomposed with some $\sigma \in G_K$ (Proposition \ref{P1} $(2)$), and it then suffices to precompose accordingly $R^+ \to C$ with $\sigma$. \qed

   \subsection{} By iterated application of Cohen factorizations, any finite sequence  $R_0\stackrel{{f}_1}{\to} R_1\stackrel{{f}_2}{\to}  \cdots \stackrel{{f}_{n }}{\to}  R_n  $  of local homomorphisms fits into a commutative diagram 
          \begin{equation}\label{CC4}   \xymatrix @-1pc { R_0   \ar[d]   \,  \ar@{->}[r]^{{f_1} }   & R_1\ar[d]   \,  \ar@{->}[r]^{{f_2} }    & \cdots   \,  \ar@{->}[r]^{{f_{n }} }    & R_n   \ar[d]    
          \\ R^\cdot_0  \,  \ar@{->}[r]^{{f^\cdot_{1 }} }     & R^\cdot_1  \,  \ar@{->}[r]^{{f^\cdot_{2 }} }     & \cdots   \,  \ar@{->}[r]^{{f^\cdot_{n }} }     & R^\cdot_n     
      }   \end{equation} 
 where the vertical maps are flat (the last one being $\rm{id}_{R_n}$), and the $f^\cdot_i$ are surjective. 
     Combining this with Theorem \ref{T5}, one obtains the following generalization of \ref{T2}:  
    
 \begin{thm}\label{T6} Any finite sequence  $R_0\stackrel{{f}_1}{\to} R_1\stackrel{{f}_2}{\to}  \cdots \stackrel{{f}_{n }}{\to}  R_n  $  of local homomorphisms of complete noetherian  local domains, with $R_0$ of mixed characteristic $(0,p)$,  fits into a commutative square  
       \begin{equation}\label{CC5}   \xymatrix @-1pc { R_0   \ar[d]   \,  \ar@{->}[r]^{{f_1} }   & R_1\ar[d]   \,  \ar@{->}[r]^{{f_2} }    & \cdots   \,  \ar@{->}[r]^{{f_{n }} }    & R_n   \ar[d]  
        \\  \;\, R_0^+  \ar[d]   \,  \ar@{->}[r]^{{f_1}^+}   & \;\, R_1^+\ar[d]   \,  \ar@{->}[r]^{{f_2}^+}    & \cdots   \,  \ar@{->}[r]^{{f_{n }}^+}    & \;\, R_n^+ \ar[d]    
     \\   C_0  \,  \ar@{->}[r]    & C_1  \,  \ar@{->}[r]    & \cdots   \,  \ar@{->}[r]    & C_n  
         }   \end{equation}  
           where each $C_i$ is a
$\hat\sK^{\sm o}$-perfectoid Cohen-Macaulay $R_i$-algebra if $R_i$ is of mixed characteristic $(0,p)$,  or a perfect Cohen-Macaulay $R_i$-algebra if $R_i$ is of characteristic $p$.

Moreover,  the ${f}_i^+$ can be given in advance. 
   \end{thm}  
 
    Here $\hat\sK^{\sm o}$ is the perfectoid ring obtained from the Witt ring of the algebraic closure of the residue field of $R_n$ by adjoining $p^{th}$-power roots of $p$ and completing.

    \medskip
    \section*{Erratum to \cite{A1}} Almost algebra hides many subtle difficulties. We are grateful to O. Gabber for pointing the following erratum to \cite[\S 1]{A1}. 

Every almost projective $\frak A$-module $\frak P$ is flat \cite[2.4.12]{GR}. But it is not clear in general that it is almost {\it faithfully} flat if it is {\it faithful}; it is not even clear that every almost finite etale extension is almost {\it faithfully} flat (as written in \cite[rem. 1.7.2 (1), 1.8.1. (1)]{A1})\footnote{this is an important issue since the perfectoid Abhyankar lemma asserts that a certain extension is almost finite etale, whereas it is the almost {faithful} flatness of this extension which is used in its applications.}. 
 The results \cite[2.4.28 iv, v]{GR} only show that the evaluation ideal $\mathcal E_{\frak P/\frak A}\subset \frak A$ satisfies $\mathcal E_{\frak P/\frak A}\frak P = \frak P$ if $\frak P$ is {faithful}, whereas {\it faithful} flatness amounts to $\mathcal E_{\frak P/\frak A}  = \frak A\,$ (the difficulty is that it is not clear that the idempotent ideal $\mathcal E_{\frak P/\frak A}$ is generated by an idempotent almost-element). 

However, the problem disappears if $\frak P$ is of constant rank $r$, because $\mathcal E_{\frak P/\frak A}  \supset \mathcal E_{\wedge^r\frak P/\frak A}= \frak A$ (\cf \cite[proof of 4.3.8]{GR}), and more generally if $\frak P$ is of finite rank.   

Fortunately, all almost finite etale extensions occurring in the paper are of finite rank (as almost projective modules) because they occur in the context of \cite[Prop. 1.9.1 (1),(3)]{A1}, in connection with Galois extensions and subextensions. However, in the given proof that Galois extensions are of rank $r = \mid G\mid$, the argument by faithful flatness descent should be replaced by the following observation (Gabber): for any extension $\frak A\inj \frak B$ and any (almost) finite projective $\frak A$-module $\frak P$ such that $\frak P_{\frak B}$ is of rank $r$, $\frak P$ is of rank $r$. Indeed, $ \wedge^{r+1} \frak P $ injects into $ {\wedge^{r+1} \frak P_{\frak B} }= 0$, hence $\frak P$ is of finite rank; furthermore, by \cite[4.3.27]{GR}, $\frak A$ decomposes into a finite product $\prod \frak A_i$ where $\frak P_{\frak A_i}$ has rank $i$, and by tensoring with $\frak B$, one sees that $\frak P_{\frak A_i}= 0$ for $i\neq r$.

         \end{sloppypar}

 \end{document}